\documentclass[11pt]{amsart}
\usepackage[margin=1.15in]{geometry}
\usepackage{amscd,amssymb, amsmath, wasysym}
\usepackage{graphicx}
\usepackage{amsfonts}
\usepackage{mathrsfs}    
\usepackage{amsmath}    
\usepackage{amsthm}     
\usepackage{amscd}      
\usepackage{amssymb}    
\usepackage{eucal}      
\usepackage{latexsym}   
\usepackage{graphicx}   
\usepackage{verbatim}   
\usepackage[all]{xy}     

\newcounter{thmcounter}

\numberwithin{thmcounter}{section}
\numberwithin{equation}{thmcounter}

\newtheorem{theorem}[thmcounter]{Theorem}
\newtheorem{proposition}[thmcounter]{Proposition}
\newtheorem{lemma}[thmcounter]{Lemma}
\newtheorem{corollary}[thmcounter]{Corollary}

\theoremstyle{definition}
\newtheorem{definition}[thmcounter]{Definition}

\newtheorem{example}[thmcounter]{Example}

\newtheorem{remark}[thmcounter]{Remark}

\newtheorem{condition}[thmcounter]{Condition}

\newtheoremstyle{claim}{9pt}{3pt}{}{\parindent}{\bf}{.}{1em}{}

\theoremstyle{claim}

\newenvironment{namelist}[1]{%
\begin{list}{}
{
\settowidth{\labelwidth}{#1}%
\setlength{\labelsep}{0.3em}%
\setlength{\leftmargin}{\labelwidth}%
\addtolength{\leftmargin}{\labelsep}}}{%
\end{list}}

\newcommand{\nC}{\mathbb{C}}                     

\newcommand{\nP}{\mathbb{P}}                     

\newcommand{\sF}{\mathscr{F}}

\newcommand{\sO}{\mathscr{O}}                    

\newcommand{\sI}{\mathscr{I}}                    
\newcommand{\sJ}{\mathscr{J}}

\newcommand{\sM}{\mathscr{M}}
\newcommand{\sN}{\mathscr{N}}

\newcommand{\sQ}{\mathscr{Q}}

\newcommand{\mf}[1]{\mathfrak{#1}}

\DeclareMathOperator{\adeg}{adeg}                

\DeclareMathOperator{\Bl}{Bl}                    

\DeclareMathOperator{\codim}{codim}              

\DeclareMathOperator{\embdim}{embdim}            

\DeclareMathOperator{\mult}{mult}                

\DeclareMathOperator{\Supp}{Supp}                
\DeclareMathOperator{\supp}{Supp}                
\DeclareMathOperator{\sHom}{\mathscr{H}om}       
\DeclareMathOperator{\Span}{Span}                
\DeclareMathOperator{\Spec}{Spec}                

\DeclareMathOperator{\reg}{reg}                  

\DeclareMathOperator{\rank}{rank}                

\newcounter{rkcounter}             
\begin{document}

\title[Castelnuovo-Mumford regularity bounds]{Castelnuovo-Mumford regularity bounds for singular surfaces}
\author{Wenbo Niu}


\address{Department of Mathematics, Purdue University, West Lafayette, IN 47907-2067, USA}
\email{niu6@math.purdue.edu}



\begin{abstract} We prove the regularity conjecture, namely Eisenbud-Goto conjecture,  for a normal surface with rational, Gorenstein elliptic and log canonical singularities. Along the way, we bound the regularity for a dimension zero scheme by its Loewy length and for a curve allowing embedded or isolated point components by its arithmetic degree.
\end{abstract}

\maketitle

\section{Introduction}

\noindent Throughout this paper we work over the complex number field $k:=\nC$. Let $X$ be a closed subscheme of $\nP^n$ defined by an ideal sheaf $\sI_X$. $X$ is said to be $m$-regular if $H^i(\nP^n, \sI_X(m-i))=0$ for all $i>0$. The minimal such number $m$ is called Castelnuovo-Mumford regularity of $X$ and is denoted by $\reg X$. Similarly, we can define $\reg \sF$ for any coherent sheaf $\sF$ on $\nP^n$. For further reference on regularity theory we refer to the book \cite{Lazarsfeld:PosAG1}.

The main result of the paper is to prove the following regularity bound for certain singular surfaces.
\begin{theorem}\label{prop08} Let $X$ be a nondegenerate normal surface in $\nP^n$ $(n\geq 4)$  with the following singularities: rational, Gorenstein elliptic and log canonical. Then one has
$$\reg X\leq \deg X-\codim X+1.$$
\end{theorem}
\noindent For a nonsingular surface, the above bound has been proved by Pinkham \cite{Pinkham:RegBoundSur} in $\nP^4$ and $\nP^5$, and by Lazarsfeld \cite{Larz:RegSurface} in $\nP^n$ with $n\geq 4$. The regularity conjecture, which is also called Eisenbud-Goto conjecture due to \cite{Eisenbud:LinFreRes}, predicts that the above bound should work for any nondegenerate projective variety. This has been proved for curves by Gruson-Lazarsfeld-Peskine \cite{GLP} (see also \cite{Giaimo:CMRegofCurves} by Giaimo) and for nonsingular surfaces as aforementioned. Some slightly weaker regularity bounds for higher dimensional varieties were obtained by Kwak in \cite{Kwak:Reg34} and \cite{Kwak:GenEquReg}.

Our approach to prove Theorem \ref{prop08} is a combination of generic projection method used in \cite{Larz:RegSurface} and Grothendieck duality. The usage of duality trick in regularity problem was shown by Lawrence Ein in his seminar notes. It provides an alternative way  for cohomological computations in \cite{Larz:RegSurface} and perhaps could be useful in future for this type of problem. The main difficulty in our approach is to bound the regularity of fibers of a generic projection. If the surface is nonsingular, then it is classical that all fibers have length no more than $3$, which implies that the regularity of a fiber is at most $3$. However, if the surface is singular, then the length of a fiber supported at a singular point would be very large, which is the obstruction of applying Lazarsfeld's proof to singular case. Thus we were led to consider to bound the regularity for a dimension zero subscheme without using its length. Fortunately, this can be done by using Loewy length, as we showed in Theorem \ref{thm01}. In general, if a zero-dimensional scheme is nonreduced, then its Loewy length could be much smaller than its length, so it provides an efficient way to bound the regularity for such a scheme, at least for our purpose. It is worth mentioning, as pointed out by the referee, that we can also bound the regularity of a nonreduced zero-dimensional scheme $X$ by $\deg(X)-\dim\Span(X)+1$, which is better than using Loewy length in many cases.

In addition, we give a regularity bound in Theorem \ref{thm02} for a curve allowing point components, extending the classical result of Gruson-Lazarsfeld-Peskine \cite{GLP} and Giaimo \cite{Giaimo:CMRegofCurves}. This simple result shows the possibility of bounding regularity by arithmetic degrees. However, in general, regularity cannot be bounded by arithmetic degrees, even in dimension one case (see Example \ref{ex:02}). This part of discussion perhaps would be interesting to the reader.\\

{\em Acknowledgement}.  Special thanks are due to professor Lawrence Ein for his generous help and encouragement. The author's thanks also go to the referees for their nice comments and suggestions.

\section{Regularity bounds for dimension zero and one subschemes of $\nP^n$}
\noindent In this section, we first give a regularity bound for a dimension zero subscheme of $\nP^n$. Classically, the regularity of such a scheme is bounded by its degree, i.e., the length of the structure sheaf of the scheme. This bound works well if the scheme is reduced but becomes worse if the scheme is nonreduced. Thus in order to get a better bound, nilpotent elements must be considered.

\begin{definition}\label{def02} Let $(A,\mf{m})$ be a local Artinian ring. We define the Loewy length $\mu_A$ of $A$ to be the nonnegative number $\displaystyle \mu_A:=\max\{i\ |\ \mf{m}^i\neq 0\}$. If $\mf{m}=0$, i.e., $A$ is a field, then we write $\mu_A=0$. We may also write $\mu$ instead of $\mu_A$ if no confusion arise.
\end{definition}

Loewy length can be thought of as an invariant to measure the size of nilpotent elements of an Artinian ring. Generally, Loewy length is much smaller than length of an Artinian ring. Here we use Loewy length to bound the regularity of a dimension zero scheme as follows.

\begin{theorem}\label{thm01} Let $X\subset \nP^n$ be a dimension zero subscheme
supported at distinct closed points $\{p_1,...,p_t\}$. For each
$1\leq i\leq t$, set $\mu_i:=\mu_{\sO_{X,p_i}}$ to be the Loewy length of the local ring $\sO_{X,p_i}$. Then one has
$$H^1(\nP^n, \sI_X(k))=0\quad\mbox{ for \ }k\geq \mu_1+\mu_2+\cdots+\mu_t+t-1,$$
i.e., $X$ is $(\mu_1+\mu_2+\cdots+\mu_t+t)$-regular.
\end{theorem}
\begin{proof} Without loss of generality we assume that $\mu_1\geq \mu_2\geq \cdots\geq\mu_t\geq 0$. We proceed by induction on $\mu_i$. If $\mu_1=\cdots\mu_t=0$, i.e., $X$ is reduced, then the result is well-known.

For general case, let $j:=\max\{i\ |\ \mu_i\neq 0\}$, so that we can assume
$$X=\Spec A_1\oplus\cdots\oplus \Spec A_j\oplus \Spec k\oplus\cdots\oplus \Spec k,$$
where $(A_i,\mf{m}_i)$ is a nonreduced local Artinian ring for $i=1,\cdots, j$.
We defined a subscheme $X_j$ of $X$ as
$$X_j:=\Spec A_1\oplus\cdots\oplus \Spec A_j/\mf{m}^{\mu_j}_j\oplus \Spec k\oplus\cdots\oplus \Spec k.$$
Set $a:=\mu_1+\cdots+(\mu_j-1)+t-1$. Then by induction, $H^1(\nP^n, \sI_{X_j}(a))=0$. Consider an exact sequence
$$0\longrightarrow \sI_X\longrightarrow \sI_{X_j}\longrightarrow \mf{m}^{\mu_j}_j\longrightarrow 0,$$
we then deduce that
$$H^0(\nP^n,\sI_{X_j}(a+1))\stackrel{\theta_j}{\longrightarrow}\mf{m}^{\mu_j}_j\longrightarrow H^1(\nP^n,\sI_X(a+1))\longrightarrow 0.$$
Thus, all we need is to show that the morphism $\theta_j$ is surjective.

From the exact sequence
$\displaystyle 0\longrightarrow\sI_{X_j}\longrightarrow\sO_{\nP^n}\longrightarrow \sO_{X_j}\longrightarrow 0$,
we have
$$0\longrightarrow H^0(\nP^n,\sI_{X_j}(a+1))\longrightarrow H^0(\nP^n,\sO_{\nP^n}(a+1))\stackrel{\phi}{\longrightarrow}\sO_{X_j}\longrightarrow 0.$$
Assume that $\mf{m}_j$ is generated by the sections $s_1,\cdots,s_e$ of $H^0(\nP^n,\sO_{\nP^n}(1))$, where $1\leq e\leq n$. Then $\mf{m}^{\mu_j}_j$ will be generated by the sections of $H^0(\nP^n,\sO_{\nP^n}(\mu_j))$ in the form
$$\sigma_{i_1\cdots i_{\mu_j}}:=s_{i_1}\cdots s_{i_{\mu_j}},\mbox{  where }1\leq i_1\leq \cdots\leq i_{\mu_j}\leq e.$$
Also for each $i\neq j$ there is a section $l_i\in H^0(\nP^n,\sO_{\nP^n}(1))$ such that $l_i\in \mf{m}_i$ but $l_i\notin \mf{m}_j$ because of the base point freeness of $\sO_{\nP^n}(1)$. Then we see that the sections in $H^0(\nP^n,\sO_{\nP^n}(a+1))$ of the form
$$s=\sigma_{i_1\cdots i_{\mu_j}}l^{\mu_1+1}_1\cdots l^{\mu_t+1}_t$$
will satisfy $\phi(s)=0$ and therefore $s\in H^0(\nP^n,\sI_{X_j}(a+1))$. Thus those sections will give the surjective morphism
$$\theta_j:H^0(\nP^n,\sI_{X_j}(a+1))\longrightarrow \mf{m}^{\mu_j}_j\longrightarrow 0.$$
This proves that $H^1(\nP^n,\sI_X(a+1))=0$.

\end{proof}

The following special case of the theorem is crucial in the proof of Theorem \ref{prop08} in section 3. So it is worth mentioning here.
\begin{corollary}{\label{s1}} Let $X\subset \nP^n$ be a subscheme supported at one point $x$ with $\mu:=\mu_{\sO_{X,x}}$. Then one has
$H^1(\nP^n, \sI_X(k))=0$ for $k\geq \mu$, i.e., $X$ is $(\mu+1)$-regular.
\end{corollary}

\begin{example}\label{ex:01} We show that the regularity bound in Theorem \ref{thm01} can be achieved. Consider a line $l$ in the projective space $\nP^n$ and $t$ reduced points $P_1,\cdots,P_t$ sitting on the line $l$. Assume that the defining ideal sheaf of $P_i$ in $\nP^n$ is $\sM_i$ for $i=1,\cdots, t$. Then we define a subscheme $Z$ defined by the ideal sheaf
$$\sI_Z:=\sM^{a_1}\cap \sM^{a_2}\cap\cdots\cap \sM^{a_t},$$
where $a_i\geq 1$ are integers for $i=1,\cdots, t$. Then it is clear that $Z$ is supported at points $P_1, \cdots, P_t$ and at each point $P_i$ the Loewy length of $\sO_{Z,P_i}$ is $a_i-1$. Thus by Theorem \ref{thm01}, we have $\reg Z\leq a_1+a_2+\cdots +a_t$. On the other hand, since the length of $l\cap Z$ is $a_1+a_2+\cdots +a_t$, i.e., $l$ is a $(a_1+a_2+\cdots+a_t)$-secant line of $Z$, we see that $\reg Z\geq a_1+a_2+\cdots+a_t$. Therefore $\reg Z=a_1+a_2+\cdots+a_t$.
\end{example}


In the last of this section, we give a regularity bound for a curve allowing points as its embedded or isolated components. Let $X$ be a closed subscheme of $\nP^n$. Denote by $R_1\sO_X$ the subsheaf of $\sO_X$ containing sections whose support has dimension $< 1$. Then $R_1\sO_X$ is naturally a $\sO_X$-ideal sheaf and it defines a subscheme $X_1$ of $X$ with the structure sheaf $\sO_{X_1}=\sO_X/R_1\sO_X$. $X_1$ is said to be obtained from $X$ by throwing away dimension zero components of $X$. We denote by $\sI^*_X$ the defining ideal sheaf of $X_1$ as a subscheme of $\nP^n$. This definition can also be found in \cite{Bayer:WhatCanComAlgGeo}. For further reference on subsheaves $R_i\sO_X$ of any $i> 0$ we refer to \cite[Chapter 2]{Hartshorne:ConnHilScheme}.

\begin{definition}\label{def01} Let $X$ be a dimension one closed subscheme of $\nP^n$ defined by an ideal sheaf $\sI_X$. We say that $X$ is a {\em curve with point components} if by throwing away its dimension zero components we obtain a reduced dimension one subscheme, i.e., the ideal sheaf $\sI^*_X$ defines a reduced equidimension one closed subscheme $X_1$ of $\nP^n$. $X_1$ is said to be the {\em curve part} of $X$.
\end{definition}

\begin{remark}
If $X$ is a curve with point components then $X$ could have embedded points or/and isolated points, or none of them which means $X$ is already a reduced curve.
\end{remark}

\begin{definition} Let $X$ be a curve with point components in $\nP^n$. We define its $0$-th arithmetic degree to be the length of $\sI^*_X/\sI_X$, i.e.,
$\adeg_0X:=l(\sI^*_X/\sI_X)$.
Define its $1$-st arithmetic degree $\adeg_1X$ to be the degree of $X_1$.
\end{definition}

\begin{lemma}\label{prop01} Let $V$ be a $k$-vector space and $\sQ$ be a coherent sheaf on $\nP^n$ with $\dim \Supp\sQ=0$ and length $l:=l(\sQ)\geq 1$. Suppose that there is a surjective morphism $V\otimes\sO_{\nP^n}\longrightarrow \sQ\longrightarrow 0$ as $\sO_{\nP^n}$-modules. Then by twisting $\sO_{\nP^n}(l-1)$ one has a surjective morphism on global sections, i.e.
\begin{equation}\label{eq01}
V\otimes H^0(\nP^n,\sO_{\nP^n}(l-1))\longrightarrow H^0(\nP^n,\sQ(l-1))\longrightarrow 0.
\end{equation}
\end{lemma}
\begin{proof} We proceed by induction on the length $l$. Starting with $l=1$ we may assume that $\sQ=k(p)$, a residue field of a closed point $p\in \nP^n$. Then the surjective morphism $V\otimes \nP^n\longrightarrow k(p)\longrightarrow 0$ shows that there is one section $s\in V$ generating $k(p)$, which means that $s$ is mapped to $1\in k(p)$. Then it is clear that the morphism in (\ref{eq01}) is surjective.

Now assume that the lemma is true for $l$, we show that it is true for $l+1$, i.e., the case that $\sQ$ has length $l+1$. Since $\dim \Supp\sQ=0$, then $\sQ$ has a submodule $k(p)$, a residue field of a closed point $p\in \Supp\sQ$. Denote by $\sQ'=\sQ/k(p)$ then $\sQ'$ is a $\sO_{\nP^n}$-module with dimension zero support and length $l$. We then have the following commutative diagram
\begin{equation}\label{eq02}
\begin{CD}
@.  @. @.0\\
@. @. @. @VVV\\
@. 0 @. @.k(p)\\
@. @VVV @. @VVV\\
0 @>>> M @>>> V\otimes\sO_{\nP^n} @>\varphi>> \sQ  @>>> 0\\
@. @VVV @| @VVV\\
0 @>>> K @>>> V\otimes\sO_{\nP^n} @>\varphi'>> \sQ' @>>> 0\\
@. @V\psi VV @. @VVV\\
  @. k(p)@. @. 0@.\\
  @. @VVV @. @.\\
  @. 0@. @. @.\\
\end{CD},
\end{equation}
where $\varphi$ is given by assumption and $\varphi'$ is induced by composing $\varphi$ with the quotient $\sQ\longrightarrow\sQ'$ and $M$ and $K$ are kernels of $\varphi$ and $\varphi'$ respectively. The left hand side vertical short exact sequence is obtained by snake lemma. Now by induction, one has a surjective morphism
$$V\otimes H^0(\nP^n,\sO_{\nP^n}(l-1))\longrightarrow H^0(\nP^n,\sQ'(l-1))\longrightarrow 0.$$
Thus it is easy to check that the coherent sheaf $K$ is $l$-regular. Therefore $K(l)$ is generated by global sections and we then have a surjective morphism
$H^0(\nP^n, K(l))\otimes \sO_{\nP^n}\longrightarrow K(l)\longrightarrow 0$. Composing with the morphism $\psi$, we obtain a commutative diagram
$$\xymatrix{
H^0(\nP^n,K(l))\otimes\sO_{\nP^n} \ar[dr]^{\varphi_p} \ar[r] & K(l) \ar[d]^{\psi} \ar[r] & 0 \\
  & k(p)(l)}.$$
The morphism $\varphi_p$ is surjective and $k(p)(l)$ has length $1$. Thus by taking global sections we have a surjective morphism $H^0(\nP^n,K(l))\longrightarrow H^0(\nP^n,k(p)(l))$ by the case $l=1$ of the induction. Now going back to the diagram (\ref{eq02}) we see $H^1(\nP^n,M(l))=0$ and therefore the morphism $V\otimes H^0(\nP^n,\sO_{\nP^n}(l))\longrightarrow H^0(\nP^n,\sQ(l))$ is surjective as required.
\end{proof}

\begin{proposition}\label{prop02} Let $X$ be a curve with point components in $\nP^n$. Assume that its curve part is $r$-regular. Then $X$ is
$$(r+\adeg_0X)\mbox{-regular}.$$
\end{proposition}
\begin{proof} Let $\sI^*_X$ be the ideal sheaf defining curve part $X_1$ of $X$. Set $l:=\adeg_0X$ which equals the length of $\sI^*_X/\sI_X$. Since $\sI^*_X$ is $r$-regular the sheaf $\sI^*_X(r)$ is generated by global sections. We then have a commutative diagram.
$$\xymatrix{
 &  & H^0(\nP^n,\sI^*_X(r))\otimes\sO_{\nP^n} \ar[d]\ar[dr]^{\varphi}  & &  \\
0 \ar[r] & \sI_X(r) \ar[r] & \sI^*_X(r) \ar[r] & \sI^*_X/\sI_X(r) \ar[r] & 0.}$$
Now applying Lemma \ref{prop01} to the morphism $\varphi$ and chasing through the diagram we then obtain the assertion.
\end{proof}

\begin{theorem}\label{thm02} Let $X\subset \nP^n$ be a curve with point components and $X_1$ be its curve part. Assume that $X_1$ is connected. Then $X$ is
$$(\adeg_0X+\adeg_1X-\dim(\Span X_1)+2)\mbox{-regular},$$
where $\Span X_1$ means the minimal linear space containing $X_1$.
\end{theorem}

\begin{proof} By the main theorem of \cite{Giaimo:CMRegofCurves}, one has $X_1$ is $(\deg X_1-\dim\Span X_1+2)$-regular. Then the result follows from Proposition \ref{prop02}.
\end{proof}

\begin{remark} (1) We can define a variety with point components of arbitrary dimension and then get a similar result as in Proposition \ref{prop02}.

(2) Theorem \ref{thm02} says that the regularity of a curve with point components can be bounded by its arithmetic degree. Notice that we require that the curve part must be reduced. It is interesting to ask if arithmetic degrees can be used to bound regularities for arbitrary schemes. Unfortunately, the answer is negative in general. In the following example, we show that even in the curve case if the curve is nonreduced then we cannot bound the regularity by arithmetic degrees.
\end{remark}

\begin{example}\label{ex:02} Consider a degree $d$ rational normal curve $X$ in $\nP^d$. Its conormal bundle is $\sN_X^*:=\sI_X/\sI^2_X=\oplus^{d-1}\sO_{\nP^1}(-d-2)$. Let $X_2$ be a subscheme defined by $\sI^2_X$. Then from a short exact sequence
\begin{equation}
0\longrightarrow \sI_X/\sI^2_X\longrightarrow \sO_{X_2}\longrightarrow \sO_X\longrightarrow 0,
\end{equation}
we have $\chi (\sO_{X_2}\otimes L^m)=\chi(\sO_X\otimes L^m)+\chi(\sN^*_X\otimes L^m)$ where $L=\sO_{\nP^d}(1)$ so that we see that $\deg X_2=d^2$. Now take a sub-line bundle $\sO_{\nP^1}(-\delta)$ of $\sN^*_X$ and let $\sJ$ be its  preimage under the quotient morphism
$$0\longrightarrow \sI^2_X\longrightarrow \sI_X\longrightarrow \sN^*_X\longrightarrow 0$$
and then we get an exact sequence
\begin{equation}
0\longrightarrow \sI^2_X\longrightarrow \sJ\longrightarrow \sO_{\nP^1}(-\delta)\longrightarrow 0.
\end{equation}
Let $Z$ be the subscheme defined by $\sJ$. We can choose the bundle $\sO_{\nP^1}(-\delta)$ to be a general subbundle of $\sN^*_X$ so that the quotient $\sI_X/\sJ=\sN^*_X/\sO_{\nP^1}(-\delta)$ is a locally free sheaf on $X$. Hence $Z$ does not have any embedded components. Then from an exact sequence
$$0\longrightarrow \sJ/\sI^2_X\longrightarrow \sO_{X_2}\longrightarrow \sO_Z\longrightarrow 0,$$
we can compute that $\deg Z=d^2-d$. Now let $\delta=dt$ for $t\gg 0$. Note that $\sO_{\nP^1}(-\delta)=\sO_{\nP^d}(-t)|_X$. Then we see that $\reg Z=\reg \sO_{\nP^1}(-\delta)=t+1$, which cannot be bounded by $\deg Z$.
\end{example}

\begin{example} The regularity bound in Theorem \ref{thm02} can be achieved. Consider a nondegenerate rational normal curve $X$ in $\nP^n$. Take a secant line $l$ of $X$ intersecting with $X$ in two distinct reduced points. Let $P_1,\cdots, P_d$ be $d$ distinct reduced points sitting on $l$ but not on $X$. Then consider a curve $Z:=X\cup \{P_1,\cdots, P_d\}$ which has those $P_i$'s as isolated components. From Theorem \ref{thm02}, we have $\reg Z\leq d+2$. On the other hand, since $l$ is a $(d+2)$-secant line of $Z$, we see $\reg Z\geq d+2$. Hence $\reg Z=d+2$.
\end{example}

\section{Regularity bounds for singular surfaces}
\noindent In this section, we prove our main theorem for a projective normal surface allowing the following  singularities: rational, Gorenstein elliptic and log canonical. For the definition and classification of those singularities we refer to \cite{Reid:ChAlgSurf}, \cite{KM98} and \cite{Matsuki:Mori}. Here we only list the results that we shall use.

\begin{theorem}\label{prop05} Let $P\in X$ be a closed point of a normal surface $X$. Denote by $\mult_PX$ the multiplicity of $X$ at $P$ and by $\embdim_PX$ the embedding dimension of $X$ at $P$.
\begin{itemize}
\item [(1)] If $P$ is a rational singular point then $\mult_PX+1=\embdim_PX$.
\item [(2)] Let $d$ be the degree of $P$ (for the definition of degree see \cite[Section 4.23]{Reid:ChAlgSurf}). If $P$ is a Gorenstein elliptic singular point, then one has
    \begin{enumerate}
\item [(i)] if $d=1,2$, then $\mult_PX=2$ and $\embdim_PX=3$;
\item [(ii)] if $d\geq 3$, then $\mult_PX=d$ and $\embdim_PX=d$.
\end{enumerate}
\item [(3)] If $P$ is a log terminal singular point then it is rational singular.
\item [(4)] If $P$ is a log canonical singular point but not log terminal and let $r$ be the index of $K_X$ at $P$, then one has
\begin{enumerate}
\item [(i)] if $r=1$, then $P$ is elliptic singular;
\item [(ii)] if $r\geq 2$, then $P$ is rational singular.
\end{enumerate}
\end{itemize}
\end{theorem}

\begin{proof} (1) is Theorem in \cite[section 4.17]{Reid:ChAlgSurf}. (2.i) is Corollary of [loc. cit., Section 4.25] and (2.ii) is Main Theorem of [loc. cit., Section 4.23]. (3) is Theorem 4-6-18 of \cite{Matsuki:Mori}. (4) is Theorem 4-6-28 of [loc. cit.].
\end{proof}

One of the classical approach to bounding regularity is using generic projection, as showed in Lazarsfeld's work \cite{Larz:RegSurface} for the nonsingular surface case. This method also has been used by Kwak to obtain regularity bounds for higher dimensional nonsingular varieties in \cite{Kwak:Reg34} and \cite{Kwak:GenEquReg}. We refer to [loc. cit.] for the details about the construction of generic projection.

Here let us focus on the surface case. Let $X$ be a nondegenerate projective surface in $\nP^n$ ($n\geq 4$). Take a linear space $\Lambda$ in $\nP^n$ of codimension $4$ disjoint with $X$. By blowing up $\nP^n$ along the center $\Lambda$ and then projecting to $\nP^3$,
we obtain the following diagram
$$
\begin{CD}
 \Bl_{\Lambda}\nP^n @>q>> \nP^3\\
 @VpVV  \\
 \nP^n.
\end{CD}
$$
Denote by $f:X\longrightarrow \nP^3$ the corresponding linear projection of $X$ to $\nP^3$ determined by the center $\Lambda$. Consider the morphism
$q_*(p^*\sO_{\nP^n}(2))\longrightarrow q_*(p^*\sO_X(2))$ induced by the restriction morphism $\sO_{\nP^n}(2)\longrightarrow \sO_X(2)$. Since $q_*(p^*\sO_X(2))=f_*\sO_X(2)$, we get a morphism
$$w_2:q_*(p^*\sO_{\nP^n}(2))\longrightarrow f_*\sO_X(2).$$
We choose the coordinates of $\nP^n$ as $T_0,\cdots,T_n$ such that  $\Lambda$ is defined by the linear forms $T_0=T_1=T_2=T_3=0$. Denote by
$V=<T_4,\cdots,T_n>$ the subspace of the vector space $H^0(\nP^n,\sO_{\nP^n}(1))$. Then we can identify
$$q_*(p^*\sO_{\nP^n}(2))=S^2V\otimes\sO_{\nP^3}\oplus V\otimes\sO_{\nP^3}(1)\oplus\sO_{\nP^3}(2),$$
where $S^2V$ is the second symmetric power of $V$.

When $X$ has isolated singularities we can further choose the center $\Lambda$ to be general so that each singular point of $X$ is the only point in a fiber of the projection $f$. We state this fact in the following lemma.

\begin{lemma}\label{prop04} Suppose that $X$ is a surface with only isolated singularities at points $\{P_1,\cdots,P_r\}$. Then for a general center $\Lambda$, the projection $f:X\longrightarrow\nP^3$ satisfies the condition that  $$\Supp f^{-1}(f(P_i))=\{P_i\},\quad\mbox{for }i=1,\cdots, r.$$
\end{lemma}
\begin{proof} Let $P$ be a singular point of $X$. We define $\Sigma_P$ to be the algebraic set swaped by the line connecting $P$ and any other point $Q$ of $X$. It is clear that the dimension of $\Sigma_P$ is at most $3$. Thus the general $\Lambda$ would not touch $\Sigma_P$ for any singular point $P$ since  $\Lambda$ has codimension $4$. Then the projection $f$ determined by $\Lambda$ will have the desired property.
\end{proof}

In the sequel, we always assume that the center $\Lambda$ is general so that the projection $f$ has the property in Lemma \ref{prop04}. The key point to obtain a regularity bound for $X$ is to prove that the morphism $w_2$ is surjective. If $X$ is nonsingular then the classical result on generic projection says that each fiber of $f$ has length no more than three and then by base change the morphism $w_2$ is surjective. This is how generic projection was used in the work \cite{Larz:RegSurface}. However, if $X$ has singular points, then the fiber of $f$ would become complicated and its length could be very large. Indeed, the main difficulty only appears, after applying Lemma \ref{prop04}, in the fiber which contains the singular points of $X$. For a fiber supported in the smooth locus of $X$, a modification of the classical result of smooth case will still give us the fact that the length of this fiber is no more than three. We give a proof for this fact by the dimension counting techniques, following the argument in \cite{GH:PrinAlgGeom} or \cite{Moishenzon:ComplexSurfaces}.

\begin{lemma}\label{p:100} Suppose $X$ is a surface with only isolated singularities at points $\{P_1,\cdots,P_r\}$. Let $X'$ be the image of a generic projection $f:X \longrightarrow \nP^3$. Write $U'=X'-\{f(P_1),\cdots, f(P_r)\}$. Then $U'$ has only ordinary singularities.
\end{lemma}
\begin{proof} For the definition of ordinary singularities, we refer to \cite[p.616]{GH:PrinAlgGeom}. The proof is almost the same as \cite[p.611]{GH:PrinAlgGeom} and \cite{Moishenzon:ComplexSurfaces}, so we only write the crucial steps here.

Write $U=X-\{P_1,\cdots,P_r\}$ the smooth locus of $X$. Since the secant variety of $X$ has dimension no more than $5$, we can certainly use a generic projection to project $X$ into $\nP^5$ such that the image of $X$ has only isolated singularities and $U$ is projected isomorphically into $\nP^5$. Hence we can assume that $X\subset \nP^5$.

Following the notation in \cite[p.611]{GH:PrinAlgGeom}, let $G(2,6)$ be the Grassmannian of lines in $\nP^5$. In the variety $X\times X\times G(2,6)$, consider the incident correspondence $I$ defined as
$$I=\{(p,q,\Lambda)\ |\ \dim \overline{pq\Lambda}\leq 2\}.$$
Then $\dim I=9$ and a fiber of $I\rightarrow G$ over a general point $\Lambda$ determines a double curve $C_{\Lambda}$ on $X$ which contains all multi-points fibers of the projection $f$ from the center $\Lambda$. The key is that we can choose $\Lambda$ general such that the corresponding double curve $C_{\Lambda}$ avoids all singular points of $X$. This is because if we fix a singular point $P$, then for any other point $Q\in X$  the lines touching $\overline{PQ}$ are parameterized by a Schubert cycle $\sigma_3$ of dimension $5$ in $G(2,6)$. Thus the subset $I_P=\{(P,Q,\Lambda)\ |\ \dim \overline{PQ\Lambda}\leq 2\}$ of $I$ has dimension $7$, which is smaller than $\dim G(2,6)$.

Once a double curve $C_\Lambda$ for a general $\Lambda$ avoids the singular points of $X$, then the rest argument is the same as \cite{GH:PrinAlgGeom}. Alternatively, we can use the technique in \cite{Moishenzon:ComplexSurfaces}. Let $T(X)$ be the closure of tangent planes at nonsingular points of $X$. Then $\dim T(X)\leq4$. The general $\Lambda\in G(2,6)$ will cut $T(X)$ at finitely many points, say $\{Q_1,\cdots, Q_s\}$. Each $Q_i$ determines some points $Q_{ij}\in X$ such that $Q_i$ is in the tangent space of $X$ at $Q_{ij}$. Since $\Lambda$ is general, we can assume all $Q_{ij}$ are in the smooth locus $U$ of $X$, (as long as $\Lambda$ does not touch the edge of the $T(X)$ as a closure of an open set). This fact plus the choice of double curve $C_{\Lambda}$ guarantees that we can take a small open neighborhood $U_i$ for each singular point $P_i$ such that under a general projection $f$ from $\Lambda$, the restriction $f|_{U_i}$ is injective and $f|_{U_i-\{P_i\}}$ is isomorphic. Now those $U_i$ will replace $U_i$ in (1) of \cite[Theorem 3, p.60]{Moishenzon:ComplexSurfaces} and the rest of the argument is the same as (2) of [loc. cit.].
\end{proof}

Lemma \ref{prop04} and Lemma \ref{p:100} show that in order to make $w_2$ surjective, we need to carefully analyze the singularities of the fibers supported at the singular points of $X$. Thus we impose the following reasonable condition on singular points.

\begin{condition}\label{condition} Let $X$ be a surface with isolated singularities in $\nP^n$. For each singular point $P$ of $X$, the Loewy length $\mu$ of the local ring $\sO_{X,P}/(l_1,l_2,l_3)$ is no more than $2$, where $l_1,l_2,l_3$ are three general linear forms of $\nP^n$ passing through the point $P$.
\end{condition}

\begin{proposition} Let $X$ be a surface with isolated singularities satisfying Condition \ref{condition}. Then for a general center $\Lambda$ the morphism $w_2$ is surjective.
\end{proposition}
\begin{proof} Let $y\in \nP^3$ and let $L_y=q^{-1}(y)$ be the fiber of $q$ over $y$, which is a linear space of $\nP^n$ of codimension $3$. Suppose that the point $y$ is cut out by linear forms $l_1,l_2,l_3$ in $\nP^3$, then the linear space $L_y$ is cut out by the forms $l_1,l_2,l_3$ in $\nP^n$. It is clear that the fiber $X_y=f^{-1}(y)$ is the scheme-theoretical intersection $X\cap L_y$. In order to show $w_2$ is surjective, by the base change, it is enough to show the surjectivity of the morphism $$w_{2,y}:H^0(L_y,\sO_{L_y}(2))\longrightarrow H^0(L_y,\sO_{X_y}(2)).$$

Since $\Lambda$ is general so the projection $f$ satisfies the result of Lemma \ref{prop04} and Lemma \ref{p:100}. We also assume that $y=f(P)$ for some point $P\in X$.
If $P$ is a nonsingular point then by the choice of $\Lambda$, $L_y$ will cut $X$ only in its nonsingular locus. We see that the fiber $X_y$ has the length at most $3$ and thus the morphism $w_{2,y}$ is surjective at $y$.

On the other hand, if $P$ is a singular point of $X$, then by the choice of $\Lambda$ we have $\supp X_y=\{P\}$. Denote by $A=\sO_{X,x}$ the local ring of $P$ on $X$. Locally at the point $P$, $l_1,l_2,l_3$ are three general elements in the maximal ideal $\mf{m}$ of $A$. It is clear that the local ring $\sO_{X_y,P}=A/(l_1,l_2,l_3)$. Since we assume that the local ring $A$ satisfies Condition \ref{condition} we see that the Lowey length of the local ring $\sO_{X_y,P}$ is no more than $2$. By Corollary \ref{s1}, $X_y$ is $3$-regular in the space $L_y$ and then the morphism $w_{2,y}$ is surjective at $y$, which finishes the proof.
\end{proof}

\begin{lemma}\label{prop06} Suppose that $P\in X$ is a normal singular point such that $\mult_PX\leq \embdim_PX$, then the local ring $\sO_{X,P}$ satisfies Condition \ref{condition}
\end{lemma}
\begin{proof} Denote by $A=\sO_{X,x}$ the local ring of $P$ on $X$. Let $l_1,l_2,l_3$ be three general elements in the maximal ideal $\mf{m}$ of $A$.  Then $(l_1,l_2)$ is a regular sequence of $A$ and is in the cotangent space. Thus if write $B=A/(l_1,l_2)$ and write $l(B)$ to be the length of $B$, we have $\mult_PX=\l(B)$ and  $\embdim(B)=\embdim_PX-2$. By assumption that $\mult_PX\leq \embdim_PX$, we see that $l(B)\leq\embdim B+2$, which implies that the Loewy length of the Artinian ring $B$ satisfies the inequality $\mu_B\leq 2$. Then the Loewy length of the local ring $A/(l_1,l_2,l_3)$ is also  $\leq 2$ since $\sO_{X_y,P}=B/(l_3)$.
\end{proof}

The following Kodaira type vanishing theorem for a projective normal surface is know to experts but we include its proof here since it is very short.

\begin{lemma} Let $X$ be a projective normal surface and $L$ be a ample line bundle on $X$. Then $H^1(X,\omega_X\otimes L)=H^2(X,\omega_X\otimes L)=0$, where $\omega_X$ is a dualizing sheaf of $X$.
\end{lemma}
\begin{proof} Let $f:X'\longrightarrow X$ be a resolution of singularities. Then we have a short exact sequence
$$0\longrightarrow f_*\omega_{X'}\longrightarrow \omega_X\longrightarrow \sQ\longrightarrow 0,$$
where $\sQ$ has the support of dimension zero since $X$ is normal. Notice that $R^if_*\omega_{X'}=0$ for $i>0$. Then by applying Kawamata-Viehweg vanishing theorem on $X'$ the result follows.
\end{proof}

Once the morphism $w_2$ is surjective, we can use Grothendieck duality to obtain a regularity bound for $X$. This part of argument is done in \cite[Section 2]{Larz:RegSurface} by cohomological computation for nonsingular case. It seems to us that Lazarsfeld's argument can also be applied here without much change. But using Duality trick provides an alternative approach which could shed new light on this type of problem.

\begin{proposition}\label{prop07} Let $X\subset\nP^n$ be a nondegenerate normal surface and suppose that $w_2$ is surjective for a general center $\Lambda$, then
$$\reg X\leq \deg X-\codim X+1.$$
\end{proposition}
\begin{proof} Recall that we choose coordinates of $\nP^n$
such that $\Lambda$ is defined by $T_0=T_1=T_2=T_3=0$ and denote
by $V=<T_4,...,T_n>$ the vector subspace of $H^0(\nP^n,\sO_{\nP^n}(1))$, then
$$q_*(p^*\sO_{\nP^n}(2))=S^2V\otimes\sO_{\nP^3}\oplus V\otimes\sO_{\nP^3}(1)\oplus\sO_{\nP^3}(2),$$
where $S^2V$ is the second symmetric power of $V$. Twisting $w_2$ by $\sO_{\nP^3}(-2)$ and writing $E$ to be the kernel, then we have an exact sequence
\begin{equation}\label{3}
0\longrightarrow E\longrightarrow S^2V\otimes\sO_{\nP^3}(-2)\oplus
V\otimes\sO_{\nP^3}(-1)\oplus\sO_{\nP^3}\longrightarrow f_*\sO_X\longrightarrow 0.
\end{equation}
Since $X$ is Cohen-Macaulay and $f$ is finite, $f_*\sO_X$ is a sheaf
of codimension one Cohen-Macaulay $\sO_{\nP^3}$-module and therefore $E$ is a locally free sheaf of rank
$$r=\rank S^2V\otimes\sO_{\nP^3}(-2)\oplus
V\otimes\sO_{\nP^3}(-1)\oplus\sO_{\nP^3}=\frac{(n-2)(n-1)}{2}.$$

We claim that the dual $E^\vee$ is $(-2)$-regular. To see this, let
$\omega_X$ be a dualizing sheaf of $X$. Applying $\sHom(\ \_\
,\omega_{\nP^3})$ to the exact sequence (\ref{3}), we get an exact sequence
\begin{equation}\label{2}
0\longrightarrow S^2V\otimes\omega_{\nP^3}(2)\oplus
V\otimes\omega_{\nP^3}(1)\oplus\omega_{\nP^3}\longrightarrow E^{\vee}(-4)\longrightarrow f_*\omega_X\longrightarrow 0.
\end{equation}
Twisting it by $\sO_{\nP^3}(1)$ and taking $H^1$ cohomology, we see that $H^1(\nP^3, E^{\vee}(-3))=0$. Then taking $H^2$ cohomology of the exact sequence (\ref{2}), we have
$$0\longrightarrow H^2(\nP^3,E^{\vee}(-4))\longrightarrow H^2(\nP^3,f_*\omega_X)\longrightarrow H^3(\nP^3,\omega_{\nP^3})\longrightarrow \cdots.$$
Since
$H^2(\nP^3,f_*\omega_X)=H^2(X,\omega_X)=H^3(\nP^3,\omega_{\nP^3})=k$,
we obtain that $H^2(\nP^3,E^{\vee}(-4))=0$. For cohomology $H^3$ of
$E^{\vee}$, twist the exact sequence (\ref{2}) by $\sO_{\nP^3}(-1)$
and then take $H^3$ cohomology to get an exact sequence
$$H^2(f_*\omega_X(-1))\stackrel{\theta}{\longrightarrow} H^3(\nP^3,V\otimes\omega_{\nP^3}\oplus\omega_{\nP^3}(-1))\longrightarrow H^3(\nP^3,E^{\vee}(-5))\longrightarrow 0.$$
We shall show that the morphism $\theta$ is surjective. By duality,
it is the same as
$$H^0(X,\sO_X(1))^{\vee}\longrightarrow H^0(\nP^3,V\otimes\sO_{\nP^3}\oplus\sO_{\nP^3}(1))^{\vee}$$
which is the dual of the morphism
$$H^0(\nP^3,V\otimes\sO_{\nP^3}\oplus\sO_{\nP^3}(1))\longrightarrow H^0(X,\sO_X(1)).$$
Note that
$H^0(\nP^3,V\otimes\sO_{\nP^3}\oplus\sO_{\nP^3}(1))=H^0(\nP^n,\sO_{\nP^n}(1))$.
Since $X$ is nondegenerate in $\nP^n$ the
morphism
$$H^0(\nP^n,\sO_{\nP^n}(1))\longrightarrow H^0(X,\sO_X(1))$$
is injective and therefore $\theta$ is surjective. Thus we obtain
$H^3(\nP^3,E^{\vee}(-5))=0$ and conclude that $E^{\vee}$ is
$(-2)$-regular.

Back to the exact sequence (\ref{3}) and let $d:=\deg X$. Since $\Supp
f_*\sO_X$ is a degree $d$ hypersurface of $\nP^3$ we obtain
\begin{eqnarray*}
c_1(E) & = & -d+c_1(S^2V\otimes\sO_{\nP^3}(-2)\oplus
V\otimes\sO_{\nP^3}(-1)\oplus\sO_{\nP^3})\nonumber\\
&=& -d-(n-1)(n-3), \nonumber
\end{eqnarray*}
and therefore $\det E=\sO_{\nP^3}(-d-(n-1)(n-3))$. Now from the
canonical identity $$E=(\wedge^{r-1}E)^{\vee}\otimes\det E,$$ we
have that $E$ is $(-2)(r-1)+d+(n-1)(n-3)$-regular, i.e.
$(d-n+3)$-regular. From the exact sequence (\ref{3}), we get that
$f_*\sO_X$ and hence $\sO_X$ is $(d-n+3)$-regular. Finally, by using
\cite[Lemma 1.5]{Larz:RegSurface}, we conclude that $\reg X=(d-n+3)$.
\end{proof}

Now combining all the things we have done, we can write down the proof of Theorem \ref{prop08}.
\begin{proof}[Proof of Theorem \ref{prop08}]This is from Theorem \ref{prop05}, Lemma \ref{prop06} and Proposition \ref{prop07}.
\end{proof}

\begin{remark} The Duality trick in the proof of Proposition \ref{prop06} can be used directly to prove a regularity bound for a nonsingular curve, which was obtained in \cite{GLP}.
\end{remark}

\bibliographystyle{alpha}

\begin{thebibliography}{GLP83}

\bibitem[BM93]{Bayer:WhatCanComAlgGeo}
Dave Bayer and David Mumford.
\newblock What can be computed in algebraic geometry?
\newblock In {\em Computational algebraic geometry and commutative algebra
  ({C}ortona, 1991)}, Sympos. Math., XXXIV, pages 1--48. Cambridge Univ. Press,
  Cambridge, 1993.

\bibitem[EG84]{Eisenbud:LinFreRes}
David Eisenbud and Shiro Goto.
\newblock Linear free resolutions and minimal multiplicity.
\newblock {\em J. Algebra}, 88(1):89--133, 1984.

\bibitem[GH94]{GH:PrinAlgGeom}
Phillip Griffiths and Joseph Harris.
\newblock {\em Principles of algebraic geometry}.
\newblock Wiley Classics Library. John Wiley \& Sons Inc., New York, 1994.
\newblock Reprint of the 1978 original.

\bibitem[Gia06]{Giaimo:CMRegofCurves}
Daniel Giaimo.
\newblock On the {C}astelnuovo-{M}umford regularity of connected curves.
\newblock {\em Trans. Amer. Math. Soc.}, 358(1):267--284 (electronic), 2006.

\bibitem[GLP83]{GLP}
L.~Gruson, R.~Lazarsfeld, and C.~Peskine.
\newblock On a theorem of {C}astelnuovo, and the equations defining space
  curves.
\newblock {\em Invent. Math.}, 72(3):491--506, 1983.

\bibitem[Har66]{Hartshorne:ConnHilScheme}
Robin Hartshorne.
\newblock Connectedness of the {H}ilbert scheme.
\newblock {\em Inst. Hautes \'Etudes Sci. Publ. Math.}, (29):5--48, 1966.

\bibitem[KM98]{KM98}
J{\'a}nos Koll{\'a}r and Shigefumi Mori.
\newblock {\em Birational geometry of algebraic varieties}, volume 134 of {\em
  Cambridge Tracts in Mathematics}.
\newblock Cambridge University Press, Cambridge, 1998.
\newblock With the collaboration of C. H. Clemens and A. Corti, Translated from
  the 1998 Japanese original.

\bibitem[Kwa98]{Kwak:Reg34}
Sijong Kwak.
\newblock Castelnuovo regularity for smooth subvarieties of dimensions {$3$}
  and {$4$}.
\newblock {\em J. Algebraic Geom.}, 7(1):195--206, 1998.

\bibitem[Kwa00]{Kwak:GenEquReg}
Sijong Kwak.
\newblock Generic projections, the equations defining projective varieties and
  {C}astelnuovo regularity.
\newblock {\em Math. Z.}, 234(3):413--434, 2000.

\bibitem[Laz87]{Larz:RegSurface}
Robert Lazarsfeld.
\newblock A sharp {C}astelnuovo bound for smooth surfaces.
\newblock {\em Duke Math. J.}, 55(2):423--429, 1987.

\bibitem[Laz04]{Lazarsfeld:PosAG1}
Robert Lazarsfeld.
\newblock {\em Positivity in algebraic geometry. {I}}, volume~48 of {\em
  Ergebnisse der Mathematik und ihrer Grenzgebiete. 3. Folge. A Series of
  Modern Surveys in Mathematics [Results in Mathematics and Related Areas. 3rd
  Series. A Series of Modern Surveys in Mathematics]}.
\newblock Springer-Verlag, Berlin, 2004.
\newblock Classical setting: line bundles and linear series.

\bibitem[Mat02]{Matsuki:Mori}
Kenji Matsuki.
\newblock {\em Introduction to the {M}ori program}.
\newblock Universitext. Springer-Verlag, New York, 2002.

\bibitem[Moi77]{Moishenzon:ComplexSurfaces}
Boris Moishezon.
\newblock {\em Complex surfaces and connected sums of complex projective
  planes}.
\newblock Lecture Notes in Mathematics, Vol. 603. Springer-Verlag, Berlin-New
  York, 1977.
\newblock With an appendix by R. Livne.

\bibitem[Pin86]{Pinkham:RegBoundSur}
Henry~C. Pinkham.
\newblock A {C}astelnuovo bound for smooth surfaces.
\newblock {\em Invent. Math.}, 83(2):321--332, 1986.

\bibitem[Rei97]{Reid:ChAlgSurf}
Miles Reid.
\newblock Chapters on algebraic surfaces.
\newblock In {\em Complex algebraic geometry ({P}ark {C}ity, {UT}, 1993)},
  volume~3 of {\em IAS/Park City Math. Ser.}, pages 3--159. Amer. Math. Soc.,
  Providence, RI, 1997.

\end{thebibliography}

\end{document}